\documentclass[a4paper,12pt]{article}

\usepackage[french, ngerman, italian, english]{babel}
\usepackage{amsmath,amsthm, amssymb}
\usepackage{setspace}
\usepackage{anysize}
\usepackage{url}
\usepackage{graphicx}
\usepackage[colorlinks=true]{hyperref}
\usepackage{mathptmx}
\usepackage{paralist}
\usepackage{verbatim}
\usepackage[utf8]{inputenc}

\theoremstyle{plain}
\newtheorem{thm}{Theorem}[section]

\newtheorem{cor}[thm]{Corollary}
\newtheorem{lem}[thm]{Lemma}
\newtheorem{prop}[thm]{Proposition}

\newtheorem{conj}[thm]{Conjecture}
\newtheorem{qs}[thm]{Question}

\theoremstyle{definition}

\newtheorem{os}[thm]{Remark}

\def\CC{{\mathbb C}}

\def\NN{{\mathbb N}}
\def\QQ{{\mathbb Q}}

\def\ZZ{{\mathbb Z}}
\def\PP{{\mathbb P}}
\def \H{{\mathcal {H}}}

\def \O{{\mathcal {O}}}

\def \V{{\mathcal {V}}}
\def \Z{{\mathcal {Z}}}
\def \aa{{\mathfrak{a}}}
\def \mm{{\mathfrak{m}}}
\def \kk{{\bar{k}}}

\def \et{{\acute{e}t}}

\def\de{\partial}
\def\ara{\operatorname{ara}}

\def\cd{\operatorname{cd}}
\def\ecd{\operatorname{\acute{e}cd}}
\def\height{\operatorname{ht}}
\def\depth{\operatorname{depth}}

\def\Grass{\operatorname{Grass}}
\def\Spec{\operatorname{Spec}}
\def\Proj{\operatorname{Proj}}

\def\codim{\operatorname{codim}}

\def\Ker{\operatorname{Ker}}
\def\chara{\operatorname{char}}

\begin{document}

\title{On the Arithmetical Rank of Certain Segre Embeddings}
\author{Matteo Varbaro\\
\footnotesize Dipartimento di Matematica\\
\footnotesize Universit\`a degli Studi di Genova, Italy\\
\footnotesize \url{varbaro@dima.unige.it}}
\date{{\small \today}} 
\maketitle

%

\begin{abstract}
\noindent We study the number of (set-theoretically) defining equations of Segre products of projective spaces times certain projective hypersurfaces, extending results by Singh and Walther. Meanwhile, we prove some results about the cohomological dimension of certain schemes. In particular, we solve a conjecture of Lyubeznik about an inequality involving the cohomological dimension and the \'etale cohomological dimension of a scheme, in the characteristic-zero-case and under a smoothness assumption. Furthermore, we show that a relationship between depth and cohomological dimension discovered by Peskine and Szpiro in positive characteristic holds true also in characteristic-zero up to dimension three. 
\end{abstract}

\section{Introduction}

The beauty of find the number of defining equations of a variety is expressed by Lyubeznik in \cite{ly} as follows:

\emph{Part of what makes the problem about the number of defining equations so interesting is that it can be very easily stated, yet a solution, in those rare cases when it is known, usually is highly non trivial and involves a fascinating interplay of Algebra and Geometry.}

In this paper we study the number of defining equations, called arithmetical rank (see Section \ref{preliminaries}), of certain Segre products of two projective varieties. 
Let us list some works that there already exist in this direction.
\begin{compactenum}
\item In their paper \cite{bruns-schwanzl}, Bruns and Schw\"anzl studied the number of defining equations of a determinantal variety. In particular they proved that the Segre product $\PP^n \times \PP^m \subseteq \PP^N$, where $N=nm+n+m$, can be defined set-theoretically by $N-2$ homogeneous equations and not less. In particular it is a set-theoretic complete intersection if and only if $n=m=1$.
\item In their work \cite{siwa}, Singh and Walther gave a solution in the case of $E \times \PP^1 \subseteq \PP^5$ where $E$ is a smooth elliptic plane curve: the authors proved that the arithmetical rank of this Segre product is $4$. Later, in \cite{song}, Song proved that the arithmetical rank of $C \times \PP^1$, where $C \subseteq \PP^2$ is a Fermat curve (i.e. a curve defined by the equation $x_0^d+x_1^d+x_2^d$), is $4$. In particular both $E\times \PP^1$ and $C\times \PP^1$ are not set-theoretic complete intersections.  
\end{compactenum}
In light of these results it is natural to study the following problem.

\emph{Let $n,m,d$ be natural numbers such that $n \geq 2$ and $m,d \geq 1$, and let $X\subseteq \PP^n$ be a smooth hypersurface of degree $d$. Consider the Segre product $Z = X \times \PP^m \subseteq \PP^N$, where $N=nm+n+m$. What can we say about the number of defining equations of $Z$?}

Notice that the arithmetical rank of $Z$ can depend, at least a priori, by invariant different from $n,m,d$: in fact we will need other conditions on $X$. However for certain $n,m,d$ we can provide some answers to this question. To this aim we will use various arguments: from complex analysis to the theory of algebras with straightening law, passing through local cohomology, \'etale cohomology and much commutative algebra.

In the case $n=2$ and $m=1$, we introduce, for every $d$, a locus of special smooth projective plane curves of degree $d$, that we will denote by $\V_d$: this locus consists in those smooth projective curves $C$ of degree $d$ which have a $d$-flex, i.e. a point $P$ at which the intersection multiplicity of $C$ and the tangent line in $P$ is equal to $d$. Using methods coming from ``algebras with straightening law's theory" we prove that for such a curve $C$ the arithmetical rank of the Segre product $C \times \PP^1 \subseteq \PP^5$ is $4$, provided that $d\geq 3$ (see Corollary \ref{ara1}). It is easy to show that every smooth elliptic curve belongs to $\V_3$ and that every Fermat's curve of degree $d$ belongs to $\V_d$, so we recover the results obtained in \cite{siwa} and in \cite{song}. However the equations that we will find are different from  those found in these papers, and our result is characteristic free. Note that a result of Casnati and Del Centina \cite{casnati-del centina} shows that the codimension of $\V_d$ in the locus of all the smooth projective plane curves of degree $d$ is $d-3$, provided that $d\geq 3$ (Remark \ref{hyperflexes}). 

For a general $n$, we can prove that if $X\subseteq \PP^n$ is a general smooth hypersurface of degree not bigger than $2n-1$, then the arithmetical rank of $X\times \PP^1\subseteq \PP^{2n+1}$ is at most $2n$ (Corollary \ref{bello}). To establish this we need a higher-dimensional version of $\V_d$. This result is somehow in the direction of the open question whether any connected projective scheme of positive dimension in $\PP^N$ can be defined set-theoretically by $N-1$ equations. 

With some similar tools we can show that, if $F=x_n^d+\sum_{i=0}^{n-3}x_iG_i(x_0,\ldots ,x_n)$ and $X=\V_+(F)$ is smooth, then the arithmetical rank of $X\times \PP^1\subseteq \PP^{2n+1}$ is $2n-1$ (Theorem \ref{45}).

Using techniques similar to those of \cite{siwa}, we are able to show the following: the arithmetical rank of the Segre product $C \times \PP^m \subseteq \PP^{3m+2}$, where $C$ is a smooth conic of $\PP^2$, is equal to $3m$, provided that $\chara(k)\neq 2$ (Theorem \ref{conic}). In particular, $C \times \PP^m$ is a set-theoretic complete intersection if and only if $m=1$.

\vspace{2mm}

Lower bounds for the arithmetical rank usually come from cohomological considerations. We collect the necessary ingredients in Section \ref{preliminaries} using results from papers of Hartshorne \cite{hartshorne4}, of Ogus \cite{ogus} and of Lyubeznik \cite{ly3} about the cohomological dimension of open subschemes of projective schemes . Actually to our purpose we could use only \'etale cohomology: in fact the results obtained in Subsection \ref{1.2} are sufficient to compute the number of defining equations of the varieties described above. However when the characteristic of the base field is $0$ it is possible to get the same lower bound (also in a more general setting) by reducing to the case when $k=\CC$ and using singular, local and sheaves cohomology. 

The results of Section \ref{preliminaries} yield some nice consequences, independently from Section \ref{upper}:
\begin{compactenum}
\item For any $n$, $m$ and $d$, if $X$ is smooth, the arithmetical rank of $X\times \PP^m\subseteq \PP^N$ can vary just among $N-2$, $N-1$ and $N$.
\item A conjecture of Lyubeznik in \cite{ly2} (see Conjecture \ref{conjly}) states, roughly speaking, that ``the \`etale cohomology provides a better lower bound for the arithmetical rank than the local cohomology". We prove the conjecture in the characteristic $0$ case under a smoothness assumption, see Theorem \ref{lyub}.
\item We extend a result by Speiser obtained in characteristic $0$ in \cite{spe}, regarding the arithmetical rank of the diagonal in $\PP^n\times \PP^n$, to any characteristic, see Corollary \ref{spei}.
\item As a consequence of Theorem \ref{depth}, we get that if a smooth projective surface $X$ has a Cohen-Macaulay homogeneous coordinate ring, then the cohomological dimension of its complement in any $\PP^{n}$ is the least possible ($\codim_{\PP^n}X-1$). In positive characteristic the analog version was proved in any dimension by Peskine and Szpiro in \cite{PS}. Instead in characteristic $0$ the statement fails already for threefolds. This fact raises a nice question about a relationship between depth and cohomological dimension (Question \ref{depth-coho}).
\end{compactenum}

\vspace{2mm}

I wish to thank in a particular way my supervisor Aldo Conca. Beyond the suggestion of this problem, he stimulated me with clever questions and discussions, which often led to interesting facts.
Moreover I want to express my gratitude to Lucian B\u adescu too, for giving me precious suggestions regarding the geometric aspect of the work.
Finally, I have to give my thanks to Ciro Ciliberto, for suggesting me Lemma \ref{cil} and its proof, and to Anurag Singh and Gennady Lyubeznik, for reading the paper and giving me useful advices.

\section{Preliminaries for the lower bounds}\label{preliminaries}

As already said in the introduction, in this section we will get the necessary lower bounds we need using results about the cohomological dimension of open subschemes of projective schemes.

First we describe in a precise way the setting in which we will work: for a noetherian ring $R$ and an ideal $I \subseteq R$ we define the arithmetical rank of $I$ with respect to $R$ as the integer
\[ \ara(I)=\min \{ k: \ \exists \ f_1, \ldots, f_k \in R \mbox{ such that } \sqrt{I}=\sqrt{(f_1, \ldots, f_k)} \}. \]
Notice that to be more precise we should write $\ara_R(I)$, however it will be always clear from the context who is $R$.
A lower bound for the arithmetical rank is given by Krull's Hauptidealsatz:
\[ \ara(I) \geq \height(I). \]
If $R$ is graded and $I$ homogeneous we can also define the homogeneous arithmetical rank, that is the integer
\[ \ara_h(I)=\min \{ k: \ \exists \ f_1, \ldots, f_k \in R \mbox{ homogeneous such that } \sqrt{I}=\sqrt{(f_1, \ldots, f_k)} \}. \]
Obviously we have
\[ \ara(I) \leq \ara_h(I). \]
Assume that $R$ is a polynomial ring of $N+1$ variables over a field $k$, and that $I$ is a homogeneous ideal of $R$. Then $\ara(I)$ gives the least number of hypersurfaces of the affine space $\mathbb{A}^{N+1}$ to define set-theoretically $\V(I)=\{\wp \in \Spec(R) \ : \ \wp \supseteq I\} \subseteq \mathbb{A}^{N+1}=\Spec(R)$; similarly $\ara_h(I)$ gives the least number of hypersurfaces of $\PP^{N}$ to intersect set-theoretically to obtain $\V_+(I)=\{\wp \in \Proj(R) \ : \ \wp \supseteq I\} \subseteq \PP^N=\Proj(R)$. It is an open problem whether these two numbers are always equal (see the survey article of Lyubeznik \cite{ly4}).

\begin{os}
The reader should be careful to the following: the number $\ara(I)$, where $I$ is an ideal of a polynomial ring, in general, does not give the minimal number of polynomials whose zero-locus is the same zero-locus of $I$, namely $\Z(I)$. For instance, if $I=(f_1,\ldots ,f_m)\subseteq \mathbb{R}[x_0,\ldots ,x_N]$, clearly 
\[ \Z(I)=\Z(f_1^2+\ldots +f_m^2).\]
However $\ara(I)$ can be bigger than $1$. The reader should keep in mind that, unless the base field is algebraically closed, there is no relations between $\V(I)$ and $\Z(I)$.
\end{os}

We will say that  $I$ (or $X=\V_+(I)$) is a {\it set-theoretic complete intersection} if $\ara_h(I)= \height(I)=\operatorname{codim}_{\PP^N} X$.

For a Noetherian ring $R$ and an ideal $I \subseteq R$ the cohomological dimension $\cd(R,I)$ of $I$ (with respect to $R$) is the supremum of the integers $i$ such that there exists an $R$-module $M$ for which $H_I^i(M)\neq 0$.
It is well known that
\[ \ara(I) \geq \cd(R,I) \geq \height(I). \]
In the same way, the cohomological dimension $\cd(X)$ of a scheme $X$ is the supremum integer $i$ such that there exists a quasi coherent sheaf $\mathcal{F}$ such that  $H^i(X,\mathcal{F})\neq 0$.\\
If $R$ is a finitely generated positively graded $k$-algebra and $I \subseteq R$ is a homogeneous ideal non-nilpotent, then
\begin{equation}\label{cd} \cd(R,I)-1= \cd(\Spec(R)\setminus \V(I))=\cd(\Proj(R)\setminus \V_+(I)) \end{equation}
(see Hartshorne \cite{hart2}); so to bound the arithmetical rank of $I$, and hence the homogeneous arithmetical rank, we will give bounds on $\cd(\Proj(R) \setminus \mathcal{V}_+(I))$.

\subsection{Bounds in characteristic 0}\label{s1}

Throughout this subsection $k$ (or $K$) will denote a field of characteristic $0$. The following remark allows us to can assume, in many cases, that the base field is $\CC$.

\begin{os}\label{0}
Let $R$ be an $A$-algebra, $\aa \subseteq R$ an ideal, $B$ a flat $A$-algebra, $R_B=R \otimes_A B$, $M$ an $R$-module and $M_B = M\otimes_A B$. Using the \u{C}ech complex it is not difficult to prove that for every $j \in \NN$:
\begin{equation}\label{redtocomplex0}
H_{\aa}^j(M)\otimes_A B \cong H_{\aa R_B}^j(M_B)
\end{equation}

Now let $S=K[x_0, \ldots, x_n]$ and $I \subseteq S$ an ideal. Since $I$ is finitely generated we can find a field $k$ such that, setting $S_k=k[x_0, \ldots, x_n]$, the following properties hold:
\[ k \subseteq K, \ \ \QQ \subseteq k \subseteq \CC, \ \ (I \cap S_k)S=I \] 
(to this aim we only have to add to $\QQ$ the coefficients of a set of generators of $I$).
Since $K$ and $\CC$ are $k$-algebras faithfully flat equation (\ref{redtocomplex0}) implies that
\begin{equation}\label{redtocomplex1}
\cd(S,I)=\cd(S_k, I \cap S_k)=\cd(S_{\CC}, (I\cap S_k)S_{\CC}),
\end{equation}
where $S_{\CC}=\CC[x_1, \ldots, x_n]$.

In the above situation assume that $I$ is homogeneous and that $X=\Proj S/I$ is smooth over $K$. Then set $X_k=\Proj (S_k / (I\cap S_k))$ and $X_{\CC}=\Proj(S_{\CC}/((I\cap S_k)\CC))$. Notice that $X \cong X_k \times _k \Spec{K}$, \ $X_{\CC}\cong X_k \times _k \Spec{\CC}$, and that $X_k$ (respective $X_{\CC}$) is smooth over $k$ (respective over $\CC$). By base change (see Liu \cite[Chapter 6, Proposition 1.24 (a)]{liu}) and by the fact that $K$ and $\CC$ are both flat $k$-algebras, we get, for all natural numbers $i,j$,
\[ H^i(X, \Omega_{X/K}^j) \cong H^i(X_k, \Omega_{X_k /k}^j) \otimes_{k}K \] 
and
\[ H^i(X_{\CC}, \Omega_{X_{\CC}/\CC}^j) \cong H^i(X_k, \Omega_{X_k /k}^j) \otimes_{k}\CC \]
(see \cite[Chapter 5, Proposition 2.27]{liu}).
Particularly we have 
\begin{equation}\label{redtocomplex2}  \dim_K(H^i(X, \Omega_{X/K}^j))=\dim_{\CC}(H^i(X_{\CC}, \Omega_{X_{\CC}/\CC}^j)) 
\end{equation}
\end{os}

In the rest of this subsection $k$ will denote a field of characteristic 0. Moreover, if $X$ is a projective variety smooth over $k$, we will write $h^{ij}(X)$ for $\dim_k(H^i(X,\Omega_{X/k}^j))$.

In the next remark, for the convenience of the reader, we collect some well known facts which we will use throughout the paper.

\begin{os}\label{betti}
Let $X$ be a projective scheme over $\CC$: we will denote $\beta_i(X)$ the topological Betti number 
\[\beta_i(X)=\operatorname{rank}_{\ZZ}(H_i^{Sing}(X_{an},\ZZ))=\operatorname{rank}_{\ZZ}(H_{Sing}^i(X_{an},\ZZ))=\]
\[=\dim_{\CC}(H_{Sing}^i(X_{an},\CC))=\dim_{\CC}(H^i(X_{an},\underline{\CC}))\] 
($X_{an}$ means $X$ regarded as a complex manifold, in the sense of Serre \cite{GAGA}, and $\underline{\CC}$ denotes the locally constant sheaf associated to $\CC$).
Pick another projective scheme over $\CC$, say $Y$, and denote by $Z$ the Segre product $X \times Y$. The K\"unneth formula for singular cohomology (for instance see Hatcher \cite[Theorem 3.16]{hatcher}) yields 
\[H_{Sing}^i(Z_{an},\CC)\cong \bigoplus_{p+q=i}H_{Sing}^p(X_{an},\CC)\otimes_{\CC}H_{Sing}^q(Y_{an},\CC),\]
thus
\begin{equation}\label{betti1}
\beta_i(Z)=\sum_{p+q=i}\beta_p(X)\beta_q(Y).
\end{equation}

Now assume that $X$ is a projective variety smooth over $\CC$. It is well known that $X_{an}$ is a K\"ahler manifold, so the Hodge decomposition (see the notes of Arapura \cite[Theorem 10.2.4]{arap}) is available. Therefore together with a theorem of Serre (see \cite[Theoreme 1, pag. 19]{GAGA}) we have
\[H_{Sing}^i(X_{an},\CC)\cong \bigoplus_{p+q=i}H^p(X_{an},(\Omega_{X/\CC})_{an}^q)\cong \bigoplus_{p+q=i}H^p(X,\Omega_{X/\CC}^q),\]
where $\mathcal{F}_{an}$ is the analyticization of a sheaf $\mathcal{F}$ (see \cite{GAGA}). Thus
\begin{equation}\label{betti2}
\beta_i(X)=\sum_{p+q=i}h^{pq}(X)
\end{equation}

Finally note that the restriction map on singular cohomology
\begin{equation}\label{injective}
H_{Sing}^i(\PP_{an}^n,\CC)\longrightarrow H_{Sing}^i(X_{an},\CC)
\end{equation}
is injective provided that $i=0, \ldots ,2 \dim X$ (see Shafarevich \cite[pp. 121-122]{sha}). In particular, since $\beta_{2i}(\PP^n)=1$ if $i\leq n$, it follows that 
\begin{equation}\label{betti3}
\beta_{2i}(X)\geq 1 \mbox{ \ \ provided that }i\leq \dim X
\end{equation}
\end{os}

The following theorem is a quite simple consequence of the results of \cite{ogus}. It provides some necessary and sufficient conditions for the cohomological dimension of the complement of a smooth variety in a projective space to be smaller than a given integer.

\begin{thm}\label{1}
Let $X\subseteq \mathbb{P}^n$ be a projective variety smooth over $k$, $r$ an integer greater than or equal to $\operatorname{codim}_{\PP^n}X$ and $U=\mathbb{P}^n \setminus X$. Then $\cd(U) < r$ if and only if
\begin{displaymath}
h^{pq}(X) = \left\{ \begin{array}{cc} 0 & \mbox{if  } p \neq q, \ p+q < n - r\\
1 &  \mbox{if  } p=q, \ p+q <n-r \end{array}  \right.
\end{displaymath}
Moreover, if $k=\CC$, the above conditions are equivalent to:
\begin{displaymath}
\beta_i(X) = \left\{ \begin{array}{cc} 1 & \mbox{if  $i< n-r$ and $i$ is even} \\
0 &  \mbox{if  $i< n-r$ and $i$ is odd} \end{array}  \right.
\end{displaymath}
\end{thm}
\begin{proof}
By equations (\ref{redtocomplex1}) and (\ref{redtocomplex2}) of Remark \ref{0} we can reduce the problem in the case in which $k=\CC$. So the ``only if"-part follows by a result of Hartshorne \cite[Corollary 7.5, p. 148]{hartshorne4}.

So it remains to prove the ``if"-part. By a theorem of Grothendieck in \cite{gro} algebraic De Rham cohomology agrees with singular cohomology. Therefore by the last part of Remark \ref{betti} the restriction maps
\begin{equation}\label{derham}
H_{DR}^i(\PP^n) \longrightarrow H_{DR}^i(X)
\end{equation}  
(where $H_{DR}$ denotes the algebraic De Rham cohomology) are injective for all $i \leq 2 \dim X$. By the assumptions, equation (\ref{betti2}) yields
$\beta_i(X)=1$ if $i$ is even and $i<n-r$, $0$ otherwise. Moreover $\beta_i(\PP^n)=1$ if $i$ is even and $i\leq 2n$, $0$ otherwise. So using again the result of Grothendieck the maps in (\ref{derham}) are isomorphisms for all $i < n-r$.

Now we would use a result of Ogus (\cite[Theorem 4.4]{ogus}), and to this aim we will show that the De Rham-depth of $X$ is greater than or equal to $n-r$. By the proof of \cite[Theorem 4.1]{ogus} this is equivalent to the fact that $\operatorname{Supp}(H_{\aa}^i(S))\subseteq \mm$ for all $i > r$, where $S=\CC[x_0, \ldots ,x_n]$, $\aa \subseteq S$ is the ideal defining $X$ and $\mm$ is the maximal irrelevant ideal of $S$. But this is easy to see, because if $\wp$ is a graded prime ideal containing $\aa$ and different from $\mm$, being $X$ non singular, $\aa S_{\wp}$ is a complete intersection in $S_{\wp}$: so $(H_{\aa}^i(S))_{\wp}\cong H_{\aa S_{\wp}}^i(S_{\wp})=0$ for all $i > r \ (\geq \operatorname{ht(\aa S_{\wp})})$. Hence \cite[Theorem 4.4]{ogus} yields the conclusion.

Finally, if $k=\CC$, the last condition is a consequence of the first one by equation (\ref{betti2}). Moreover, it implies the first one because the restriction maps of singular cohomology
\[H_{Sing}^i(\PP_{an}^n,\CC) \longrightarrow H_{Sing}^i(X_{an},\CC) \]
(that are injective if $i<n-r$ by the last part of Remark \ref{betti}) are compatible with the Hodge decomposition (see \cite[Corollary 11.2.5]{arap}).

\end{proof}

\begin{os}\label{elliptic}
Theorem \ref{1} does not hold in positive characteristic: for instance pick an elliptic curve $E$ over a field of positive characteristic whose Hasse invariant is $0$. Then set $X=E \times \PP^1 \subseteq \PP^5$ and $U=\PP^5 \setminus X$. The Frobenius acts as 0 on $H^1(X,\O_X)$, so $\cd(U)=2$ (see Hartshorne and Speiser \cite{ha-sp} or Lyubeznik \cite{ly1}). However notice that $H^1(X,\O_X)\neq 0$.
\end{os}

The two propositions below provide the necessary lower bound we need to compute the arithmetical rank of certain Segre products in characteristic 0.

\begin{prop}\label{mah}
Let $X$ and $Y$ be two positive dimensional projective schemes smooth over $k$, and set $Z=X \times Y \subseteq \PP^N$ (any embedding) and $U=\PP^N \setminus Z$. Then $\cd(U) \geq N-3$. In particular if $\dim Z \geq 3$, $Z$ is not a set-theoretic complete intersection.
\end{prop}
\begin{proof}
By equation (\ref{redtocomplex1}) we can assume $k=\CC$. Using equation (\ref{betti3}) we have $\beta_0(X) \geq 1$, \ $\beta_2(X) \geq 1$, \ $\beta_0(Y) \geq 1$ and $\beta_2(Y) \geq 1$, so equation (\ref{betti1}) implies $\beta_2(Z) \geq 2$. Now equation (\ref{betti2}) and Theorem \ref{1} yield the conclusion.
\end{proof}

\begin{os}
The proof of Proposition \ref{mah} yields the following nice fact:

Let $X$ and $Y$ be two positive dimensional projective varieties smooth over $\CC$ and $Z=X \times Y \subseteq \PP^N$. Then the dimension of the secant variety of $Z$ in $\PP^N$ is at least $2 \dim Z - 1$.

To prove this note, as in the proof of Proposition \ref{mah}, that $\beta_2(Z)\geq 2$. By a theorem of Barth (see Lazarsfeld \cite[Theorem 3.2.1]{la}), it follows that $Z$ cannot be embedded in any $\PP^M$ with $M<2 \dim X -1$. If the dimension of the secant variety were less than $2 \dim X-1$, it would be possible to project down in a biregular way $X$ from $\PP^N$ in $\PP^{2 \dim X-2}$, and this would be a contradiction.

Note that the above lower bound is the best possible, in fact $\PP^r \times \PP^s$ can be embedded in $\PP^{2(r+s)-1}$ (see Hartshorne \cite[p. 1026]{hart1}).
\end{os}

\begin{os}\label{gowa}
The statement of Proposition \ref{mah} is false in positive characteristic. To see this, consider two Cohen-Macaulay graded $k$-algebras $A$ and $B$ of negative $a$-invariant. Set $R=A\#B$ their Segre product (with the notation of the paper of Goto and Watanabe \cite{GW}). By \cite[Theorem 4.2.3]{GW} $R$ is Cohen-Macaulay as well. So, presenting $R$ as a quotient of a polynomial ring of $N+1$ variables, say $R\cong P/I$, a theorem of Peskine and Szpiro in \cite{PS} implies that $\cd(P,I)=N+1- \dim R$ (because $\chara(k)>0$). Translating in the language of Proposition \ref{mah} we have $X=\Proj(A)$, $Y=\Proj(B)$, $Z=\Proj(R) \subseteq \PP^N=\Proj(P)$ and $\cd(\PP^N \setminus Z)=\cd(P,I)-1=N - \dim Z -1$.
\end{os}

\begin{prop}\label{15}
Assume that $X$ is a projective variety smooth over $k$ such that $H^1(X,\O_X)\neq 0$ and let $Y$ be any projective scheme over $k$. As above set $Z=X \times Y \subseteq \PP^N$ (any embedding) and $U=\PP^N \setminus Z$. Then $\cd(U) \geq N - 2$.
\end{prop}
\begin{proof}
By virtue of Remark \ref{0} we may assume $k=\CC$. The assumptions imply that $\beta_1(X)\neq 0$ by equation (\ref{betti2}), and so $\beta_1(Z) \neq 0$ by equation (\ref{betti1}). Clearly $U$ is smooth, so \cite[Theorem 7.4, p. 148]{hartshorne4} implies the conclusion.
\end{proof}

\begin{os}
If in the situation of proposition \ref{15} $\dim Z\geq 2$, then it follows that $Z$ cannot be a set--theoretical complete intersection. This is a consequence of a more general result of Hartshorne obtained in \cite{hart2}, that states that an irregular projective variety $X$ over a field of characteristic $0$ (i.e. $q(X)=h^{10}(X)\neq 0$), of dimension is greater than $1$, cannot be a set-theoretical complete intersection in any $\PP^n$.
\end{os}

\subsection{Bounds in arbitrary characteristic}\label{1.2}

If the characteristic of the base field is $0$ we have seen in the previous subsection that we can, usually, reduce the problem to $k=\CC$; in this context is available the complex topology, so we can use methods from algebraic topology and from complex analysis.

Unfortunately when the characteristic of $k$ is positive, the above techniques are not available. Moreover some of the results obtained in Subsection \ref{s1} are not true in positive characteristic, as we have shown in Remarks \ref{elliptic} and \ref{gowa}. To avoid these difficulties we have to use \'etale cohomology, that gives a lower bound for the number of equations defining a variety as well as local cohomology (see equation (\ref{ecd}) of Remark \ref{etale}). This subject was introduced by Grothendieck in \cite{sga4}. Other references are the book \cite{milne} and the lectures \cite{milnel} of Milne.

For a scheme $X$ we denote by $X_{\et}$ the \'etale site of $X$ and, with a slight abuse of notation, by $\ZZ/l\ZZ$ the constant sheaf associated to $\ZZ/l\ZZ$ (for any $l \in \ZZ$). Moreover, we denote by $\ecd(X)$ the \'etale cohomological dimension of $X$, that is the largest integer $i$ such that there exists a torsion sheaf $\mathcal{F}$ on $X_{\et}$ with \'etale cohomology group $H^i(X_{\et}, \mathcal{F}) \neq 0$ ($H^i$ denotes the usual cohomology of sheaves).
Below we collect some basic results about the \'etale cohomological dimension.

\begin{os}\label{etale}
If $X$ is a $n$-dimensional scheme of finite type over a separably closed field, then $\ecd(X) \leq 2n$ (\cite[Chapter VI, Theorem 1.1]{milne}).
If moreover $X$ is affine, then $\ecd(X) \leq n$ (\cite[Chapter VI, Theorem 7.2]{milne}).

Assume that $X=\Proj(R)$ is projective and pick a closed subscheme $Y=\mathcal{V}_+(I) \subseteq X$. Then $U=X \setminus Y$ can be cover by $ \ara_h(I)$ affine subsets of $X$. Moreover the \'etale cohomological dimension of these affine subsets is less than or equal to $n$ for what said above. So, using repetitively the Mayer-Vietoris sequence (\cite[Chapter III, Exercise 2.24]{milne}), it is easy to prove that
\begin{equation}\label{ecd}
\ecd(U) \leq n+ \ara_h(I) -1
\end{equation}
The above inequality was remarked, for instance, by Newstead in \cite{new}. 
\end{os}

We recall the following result of \cite[Proposition 9.1, (iii)]{ly3}, that can be seen as an \'etale version of \cite[Theorem 8.6, p. 160]{hartshorne4}.

\begin{thm}\label{13}(Lyubeznik)
Let $k$ be a separably closed field of arbitrary characteristic, $Y \subseteq X$ two projective varieties such that $U=X \setminus Y$ is non-singular. Set $N=\dim X$,  and $l \in \ZZ$ coprime with the characteristic of $k$. If  $\ecd(U) < 2N-r$, then the restriction maps
\[H^i(X_{\et}, \ZZ/l\ZZ) \longrightarrow H^i(Y_{\et},\ZZ/l\ZZ)\] 
are isomorphism for $i < r$ and injective for $i=r$.
\end{thm}

\begin{os}
The \'etale version of Theorem \ref{1} does not hold. In fact, the integer $\ecd(\PP^N \setminus Y)$ is not an invariant of only $Y$ and $N$ (as instead is for the integer $\cd(\PP^N \setminus Y)$). For instance we can consider $\PP^2 \subseteq \PP^5$ (embedded as a linear subspace) and $v_2(\PP^2)\subseteq \PP^5$ (where $v_2(\PP^2)$ is the $2$nd Veronese embedding): the first one is defined (also scheme-theoretically) by $3$ linear equations, so $\ecd(\PP^5 \setminus \PP^2)\leq 7$ by equation (\ref{ecd}); instead, $\ecd(\PP^5 \setminus v_2(\PP^2))=8$ by \cite{barile}.  

Notice that the above argument shows that the number of defining equations of a projective schemes $X\subseteq \PP^n$ depends on the embedding, and not only on $X$ and on $n$. This suggests the limits of the use of local cohomology on certain problems regarding the arithmetical rank.
\end{os}

In \cite{spe} Speiser, among other things, computed the arithmetical rank of the diagonal $\Delta = \Delta(\PP^n)\subseteq \PP^n \times \PP^n$, provided that the characteristic of the base field is $0$. In characteristic $p>0$ he proved that the cohomological dimension of $\PP^n\times \PP^n \setminus \Delta$ is the least possible, i.e. $n-1$, therefore in positive characteristic it is not known the arithmetical rank of $\Delta$. Actually Theorem \ref{13} easily implies that the result of Speiser holds in arbitrary characteristic, since the upper bound found in \cite{spe} is valid in arbitrary characteristic. However, since in that paper the author did not describe the equations needed to define set-theoretically $\Delta$, we provide the upper bound with a different method, that yields an explicit set of equations for $\Delta$.

To this aim, we recall that the coordinate ring of $\PP^n \times \PP^n$ is $A=k[x_iy_j:i,j=0, \ldots ,n]$ and the ideal $I\subseteq A$ defining $\Delta$ is $I=(x_iy_j-x_jy_i:0\leq i<j\leq n)$.  

\begin{cor}\label{spei}
In the situation described above (with $k$ a separably closed field of arbitrary characteristic) $\ara_h(I)=2n-1$.
\end{cor}
\begin{proof}
As already said, by \cite[Proposition 2.1.1]{spe} we already know that $\ara_h(I)\leq 2n-1$. However we can observe that, if we consider $IR\subseteq R=k[x_i,y_j:i,j=0, \ldots ,n]$, then $IR$ is the ideal generated by the 2-minors of the $2\times (n+1)$ matrix of indeterminates whose rows are, respectively, $x_0, \ldots ,x_n$ and $y_0,\ldots ,y_n$. So, by \cite[(5.9) Lemma]{bruns-vetter}, a set of generators of $IR\subseteq R$ up to radical is
\[g_k = \sum_{{0\leq i<j\leq n}\atop {i+j=k}}(x_iy_j-x_jy_i), \ \ \ k=1, \ldots ,2n-1.\]
Since these polynomials belong to $A$ and since $A$ is a direct summand of $R$, we get 
\[ \sqrt{(g_1, \ldots ,g_{2n-1})A}=I, \]
therefore $\ara_h(I)\leq 2n-1$.

For the lower bound choose $l$ coprime with $\chara(k)$. K\"unneth formula for \'etale cohomology (\cite[Chapter VI, Corollary 8.13]{milne}) implies that 
\[H^2(\PP_{\et}^n\times \PP_{\et}^n, \ZZ/l\ZZ) \cong (\ZZ/l \ZZ)^2,\]
while $H^2(\Delta_{\et} , \ZZ/l\ZZ)\cong H^2(\PP_{\et}^n , \ZZ/l\ZZ)\cong \ZZ/l\ZZ$. So Theorem  \ref{13} yields $\ecd(U)\geq 4n-2$, where $U=\PP^n \times \PP^n \setminus \Delta$. Therefore equation (\ref{ecd}) yields the conclusion.
\end{proof}

The next two propositions are the analogue of Propositions \ref{mah} and \ref{15}. We need them to compute the homogeneous arithmetical rank of certain Segre products in arbitrary characteristic. First we need a remark:

\begin{os}
Let $X$ be a projective variety smooth over a field $k$ and $l$ an integer coprime to $\chara(k)$. The kernel of the cycle map is contained in the kernel of the projection from the Chow ring to itself modulo numerical equivalence. But this last group is non-zero because $X$ is projective, so we have
\begin{equation}\label{bettietale}
H^{2i}(X_{\et},\ZZ_l) \neq 0 \ \ \ \forall \ i=0, \ldots , \dim X .
\end{equation}
Therefore there exists an integer $n$ such that $H^{2i}(X_{\et},\ZZ/l^n\ZZ)$ is non-zero for any $i=0,\ldots ,\dim X$.
See the proof of \cite[Chapter VI, Theorem 11.7]{milne}.
\end{os}

\begin{prop}\label{14}
Let $k$ an algebraic closed field of arbitrary characteristic. Let $X$ and $Y$ be two projective varieties smooth over $k$ of dimension at least 1. Set $Z=X \times Y \subseteq \PP^N$ (any embedding) and $U=\PP^N \setminus Z$. Then $\ecd(U) \geq 2N-3$. In particular if $\dim Z \geq 3$, $Z$ is not a set-theoretic complete intersection.
\end{prop}
\begin{proof}
By the above remark there is an integer $l$ coprime with $\chara(k)$ such that the modules $H^i(X_{\et},\ZZ/l \ZZ)$ and $H^i(Y_{\et},\ZZ/l \ZZ)$ are non-zero $\ZZ/l\ZZ$-modules. But $H^2(\PP_{\et}^N,\ZZ/l \ZZ)$ $\cong$ $ \ZZ/l\ZZ$, therefore by K\"unneth formula for \'etale cohomology (\cite[Chapter VI, Corollary 8.13]{milne}) it follows that $H^2(Z_{\et},\ZZ/l \ZZ)$ cannot be isomorphic to $H^2(\PP^N_{\et},\ZZ/l \ZZ)$. Now Theorem \ref{13} implies the conclusion.
\end{proof}

\begin{prop}\label{curve}
Let $k$ an algebraically closed field, $C$ a smooth projective curve of positive genus, $X$ a projective scheme and $Y=C \times X \subseteq \PP^N$ (any embedding). Then $\ecd(\PP^N \setminus Y) \geq 2N-2$. In particular, if $\dim X \geq 1$, then $Y$ is not a set-theoretic complete intersection.
\end{prop}
\begin{proof}
Set $g$ the genus of $C$. By \cite[Proposition 14.2 and Remark 14.4]{milnel} it follows that $H^1(C_{\et},\ZZ/l\ZZ)\cong (\ZZ/l\ZZ)^{2g}$. Moreover $H^0(X_{\et},\ZZ/l \ZZ)\neq 0$ and $H^1(\PP^N_{\et},\ZZ/l \ZZ)=0$. But by K\"unneth formula for \'etale cohomology $H^1(Y_{\et},\ZZ/l \ZZ)\neq 0$, therefore Theorem \ref{13} let us conclude.
\end{proof}

\subsection{Consequences on the cohomological dimension}

In this subsection we draw two nice consequences of the investigations we made in the first part of the work. They are in the direction of a problem stated by Grothendieck, who asked in \cite[p. 79]{gro1} to find conditions, fixed a positive integer $c$, under which $\cd(R,I)\leq c$, where $I$ is an ideal in a ring $R$.

The first fact we want to present is a consequence of Theorem \ref{1}, and regards a relationship between cohomological dimension of an ideal in a polynomial ring  and the depth of the relative quotient ring.
It was proved by Peskine and Szpiro in \cite{PS} that if $I \subseteq S=k[x_1, \ldots ,x_n]$ is a homogeneous ideal of a polynomial ring over a field of positive characteristic such that $\depth(S/I)\geq t$, then $\cd(S,I)\leq n-t$. The same assertion does not hold in characteristic $0$, in fact examples are known already for $t=4$ (for instance if $I$ defines the Segre product of two projective spaces). When $t=2$ the statement is true also in characteristic $0$ by a result of Hartshorne and Speiser (for instance see \cite{hartshorne4}). We can settle the case $t=3$ in the smooth case.

\begin{thm}\label{depth}
Let $S=k[x_1,\ldots ,x_n]$ be a polynomial ring over a field of characteristic $0$. If $I\subseteq S$ is a homogeneous prime ideal such that $(S/I)_{\wp}$ is a regular local ring for any homogeneous prime ideal $\wp \neq \mm=(x_1, \ldots ,x_n)$ and such that $\depth(S/I)\geq 3$, then $\cd(S,I)\leq n-3$.
\end{thm} 
\begin{proof}
Suppose by contradiction that $\cd(S,I)\geq n-2$. Set $X=\Proj(S/I)\subseteq \PP^{n-1}=\Proj(S)$. So we are supposing that $\cd(\PP^{n-1}\setminus X)\geq n-3$ by equation (\ref{cd}). By the assumptions $X$ is a projective variety smooth over $k$, therefore Theorem \ref{1} implies that $h^{10}(X)\neq 0$ or that $h^{01}(X)\neq 0$. But with the notation of Remark \ref{0}, $h^{10}(X)=h^{10}(X_{\CC})$ and $h^{01}(X)=h^{01}(X_{\CC})$. So, since $h^{10}(X_{\CC})=h^{01}(X_{\CC})$ (using \cite[Theorem 10.2.4]{arap} and \cite{GAGA} together), we have $h^{10}(X)\neq 0$. But $H^1(X,\O_X)=[H_{\mm}^2(S/I)]_0 \subseteq H_{\mm}^2(S/I)$ ($[ \ ]_0$ denotes the $0$-degree part), so $\depth(S/I)\leq 2$, that is a contradiction.
\end{proof}

Actually the cited result of Peskine and Szpiro holds true whenever the ambient is a regular local ring of positive characteristic. Moreover, one can easily deduce by the result of Huneke and Lyubeznik \cite[Theorem 2.9]{hu-ly} the following: If $R$ is an $n$-dimensional regular local ring containing its residue field and $\aa\subseteq R$ is an ideal such that $\depth(R/\aa)\geq 2$, then $\cd(R,\aa)\leq n-2$. Together with these facts, Theorem \ref{depth} raises the following question:

\begin{qs}\label{depth-coho}
Suppose that $R$ is a regular local ring, and that $I\subseteq R$ is an ideal such that $\depth(R/I)\geq 3$. Is it true that $\cd(R,I)\leq \dim R -3$?
\end{qs}

The second fact we want to show is a consequence of Theorem \ref{1} and Theorem \ref{13}. It provides a solution of a special case of a conjecture stated by Lyubeznik in \cite[Conjecture, p. 147]{ly2}:

\begin{conj}(Lyubeznik)\label{conjly}
If $U$ is a $n$-dimensional scheme of finite type over a separably closed field, then $\ecd(U)\geq n+\cd(U)$.
\end{conj}

\begin{thm}\label{lyub}
Let $X \subseteq \PP^n$ be a projective variety smooth over $\CC$, and $U=\PP^n \setminus X$. Then
\[\ecd(U)\geq n + \cd(U)\]
\end{thm}
\begin{proof} 
Set $\cd(U)=s$, and define an integer $\rho_s$ to be $0$ (resp. $1$) if $n-s-1$ is odd (resp. if $n-s-1$ is even). By Theorem \ref{1} and by equation (\ref{betti3}), it follows that $\beta_{n-s-1}(X)>\rho_s$. Consider, for a prime number $p$, the $\ZZ/p\ZZ$-vector space $\operatorname{Hom}_{\ZZ}(H_i^{Sing}(X_{an},\ZZ), \ZZ/p\ZZ)$. Since $H_i^{Sing}(X_{an},\ZZ)$ is of rank bigger than $\rho_s$, the above $\ZZ/p\ZZ$-vector space has dimension greater than $\rho_s$. Therefore by the surjection given by the universal coefficient theorem 
\[ H_{Sing}^{n-s-1}(X_{an},\ZZ/p\ZZ) \longrightarrow \operatorname{Hom}_{\ZZ}(H_{n-s-1}^{Sing}(X_{an},\ZZ), \ZZ/p\ZZ)\] 
(see \cite[Theorem 3.2, p. 195]{hatcher}), 
we infer that $ \dim_{\ZZ/p\ZZ}H_{Sing}^{n-s-1}(X_{an},\ZZ/p\ZZ) > \rho_s$.
Now a comparison theorem due to Grothendieck (see \cite[Chapter III, Theorem 3.12]{milne}) yields 
\[ \dim_{\ZZ/p\ZZ}H^{n-s-1}(X_{\et},\ZZ/p\ZZ)> \rho_s .\] 
Since $\dim_{\ZZ/p\ZZ}(H^{n-s-1}(\PP_{\et}^n,\ZZ/p\ZZ))=\rho_s$, Theorem \ref{13} implies that $\ecd(U) \geq 2n- (n-s)=n+s$.
\end{proof}

Theorem \ref{lyub} might look like a very special case of Conjecture \ref{conjly}. However the case when $U$ is the complement of a projective variety in a projective space is a very important case. In fact the truth of Conjecture \ref{conjly} would ensure that to bound the homogeneous arithmetical rank from below it would be enough to work just with \`etale cohomology, and not with sheaf cohomology. Since usually one is interested in computing the number of (set-theoretically) defining equations of a projective variety in the projective space, in some sense the most interesting case of Conjecture \ref{conjly} is when $U=\PP^n\setminus X$ for some projective variety $X$. From this point of view, one can look at Theorem \ref{lyub} in the following way: {\it In order to give a lower bound for the minimal number of hypersurfaces of $\PP_{\CC}^n$ cutting out set-theoretically a smooth projective variety $X\subseteq \PP_{\CC}^n$, it is better to compute $\ecd(\PP_{\CC}^n\setminus X)$ than $\cd(\PP_{\CC}^n \setminus X)$.}

Unfortunately, Lyubeznik informed the author of this paper by a personal communication that he found a counterexample, yet unpublished, to Conjecture \ref{conjly} when the characteristic of the base field is positive: his counterexample consists in a scheme $U$ which is the complement in $\PP^n$ of a reducible projective scheme.

\section{Upper bounds}\label{upper}

In this section finally we present the defining equations of the varieties described in the introduction. The main tools we use come from ASL theory.

\subsection{Notation}

We want to fix some notation that we will use throughout this section.
Let $k$ be a field of arbitrary characteristic.

We recall that the Segre product of two finitely generated $\NN$-graded $k$-algebra $A$ and $B$ is defined as
\[ A  \sharp B = \bigoplus_{n \in \NN} A_n \otimes_k B_n. \]
This is a $\NN$-graded $k$-algebra and it is a direct summand of the tensor product $A \otimes_k B$.


Fix $n,m$ integers greater than or equal to 1. Then $X \subseteq \PP^n$ and $Y \subseteq \PP^m$ will always denote two projective schemes defined respectively by the standard graded ideals $\aa \subseteq R=k[x_0, \ldots, x_n]$ and $\mathfrak{b}\subseteq S=k[y_0, \ldots, y_m]$. 


Consider the Segre product $Z=X \times Y$ and set $A=R/\aa$ and $B=S/\mathfrak{b}$. Then we have that $Z \cong \Proj (A \sharp B)$. Moreover, if $W:= k[x_iy_j : i=0, \ldots ,n; \ \ j=0, \ldots, m] \subseteq k[x_0, \ldots, x_n, y_0, \ldots , y_m]=R \otimes_k S$, then $A  \sharp B = W/\mathcal{I}$
with $\mathcal{I}\subseteq W$ an homogeneous ideal.
Assuming that $\aa=(f_1, \ldots, f_r)$ and $\mathfrak{b}=(g_1, \ldots, g_s)$ with $\deg f_i=d_i$ and $\deg g_j=e_j$, it is easy to see that $\mathcal{I}$ is generated by the following polynomials:
\begin{compactenum}
\item $M \cdot f_i$ where $M$ varies among the monomials in $S_{d_i}$ for every $i=1, \ldots, r$;
\item $g_j \cdot N$ where $N$ varies among the monomial in $R_{e_j}$ for every $j=1, \ldots, s$.
\end{compactenum}

\vspace{2mm}

Now we present $A \sharp B$ as a quotient of a polynomial ring. So consider $P=k[z_{ij}:i=0, \ldots,n: \ j=0, \ldots,m]$ and the $k$-algebra homomorphism $\phi: P \longrightarrow A \sharp B$
defined as $\phi= \phi' \circ \pi$ where $\phi':P \longrightarrow W$ maps $z_{ij}$ to $x_i y_j$
and $\pi:W \longrightarrow A \sharp B \cong W/\mathcal{I}$ is the projection. Therefore set $I= \Ker \phi$.
With this notation, then, we have
\[ V \cong \Proj (P/I) \subseteq \PP^N, \ \ \ \ N=nm+n+m \]


Now we describe a system of generators which we will use in this section.
For all monomials $M \in S_{d_i}$ (where $i=1, \ldots,r$), choose a polynomial $(f_i)_{M} \in P$ such that $\phi'((f_i)_{M})=M \cdot f_i$.
in the same way choose a polynomial $(g_j)_{N} \in P$ for all monomials $N \in R_{e_j}$ and $j=1, \ldots,s$. Then
it is easy to show that
\[ I=I_2(Z)+J \]
where
\begin{compactenum}
\item $I_2(Z)$ denotes the ideal generated by the 2-minors of the matrix $Z=(z_{ij})$;
\item $J=((f_i)_{M}, (g_j)_{N}: \mbox{ for all }i=1, \ldots, r \mbox{ and for all monomials }M \in S_{d_i},  \mbox{ for} \\ \mbox{all } j=1, \ldots , s \mbox{ and for all monomials } N \in R_{e_j})$.
\end{compactenum}


\subsection{The defining equations of certain Segre products}

Our purpose is to exhibit a minimal set of defining equations (up to radical) for $I$ in $P$, and so to compute the arithmetical rank of $I$. We are able to solve this problem for certain ideals $\aa$ and $\mathfrak{b}$.





We need the following remark to work with algebraically closed fields and to use the Nullstellensatz:

\begin{os}\label{7}
Let $H$ be a $k$-algebra and $\mathfrak{h} \subseteq H$ an ideal. Set $H_{\bar{k}}=H \otimes_k \bar{k}$ and $\mathfrak{h}_{\kk}=\mathfrak{h} H_{\kk} \subseteq H_{\kk}$, where $\kk$ denotes the algebraic closure of $k$. Because $\kk$  is faithfully flat over $k$, if $h_1, \ldots, h_t \in \mathfrak{h}$ are such that $\sqrt{\mathfrak{h}_{\kk}}=\sqrt{(h_1, \ldots, h_t)H_{\kk}}$, then $\sqrt{\mathfrak{h}}=\sqrt{(h_1, \ldots, h_t)}$.  
\end{os}

In the following remark we make use of an argument that we will be used several times later on.
\begin{os}\label{8}
Actually the described generators of $I$ are too much: for instance for a polynomial $f_i$ of the starting ideal we have to consider all the polynomials $(f_i)_M$ with $M$ varying in $S_{d_i}$. These are $\binom{m+d_i}{m}$ polynomials! Anyway, up to radical, it is enough to choose $m+1$ monomials for every $f_i$ and $n+1$ monomials for every $g_j$.

For every $i= 1,\ldots, r$ and $l =0, \ldots, m$, set $M=y_l^{d_i}$. A possible choice for $(f_i)_M$ is $(f_i)_l:=f_i(z_{0l}, \ldots, z_{nl}) \in P$. In the same manner for every $j= 1,\ldots, s$ and $k =0, \ldots, n$ we define $(g_j)_k=g_j(z_{k0}, \ldots, z_{km}) \in P$. We claim that
\[ \sqrt{I}=\sqrt{I_2(Z)+J'} \]
where $J'$ is the ideal generated by the $(f_i)_l$'s and the $(g_j)_k$'s.

We can assume that $k$ is algebraically closed by Remark \ref{7}. So, denoting by $\Z(L)$ the zeroes locus of an ideal $L$, it is enough to prove that $\Z(I)=\Z(I_2(Z)+J')$ by Nullstellensatz.
So pick $P=[P_{00},P_{10}, \ldots ,P_{n0},P_{01}, \ldots, P_{n1}, \ldots ,P_{0m}, \ldots , P_{nm}] \in \mathcal{Z}(I_2(Z)+J)$. We can write $P=[P_0, \ldots, P_m]$, where $P_h=[P_{0h}, \ldots, P_{nh}]$ is $[0,0, \ldots ,0]$ or a point of $\PP^n$. Since $P \in \Z(I_2(Z))$ it follows that the non-zero points among the $P_h$'s are equal as points of $\PP^n$. Moreover, if $P_l$ is a non-zero point, $(f_i)_l(P)=0$ for all $i=1, \ldots ,r$ means that $P_l \in X$: then from the discussion above trivially $(f_i)_M(P)=0$ for every $i, M$ and any choice of $(f_i)_M$. By symmetry one can prove that also all the $(g_j)_N$'s vanish on $P$, so we conclude. 
\end{os}

\begin{os}\label{9}
Assume that $X=\mathcal{V}_+(F) \subseteq \mathbb{P}_k^n$ is a projective hypersurface ($F=f_1$), $m=1$ and $Y=\PP^1$. We already know from a general theorem of Eisenbud and Evans (see \cite[Theorem 2]{eisenbudevans}) that
\[ \ara(I) \leq \ara_h(I) \leq N=2n+1. \]
In this case we can find an explicit set of polynomials which generate $I$ up to radical. In fact, from a theorem of Bruns and Schw\'anzl (see \cite[Theorem 2]{bruns-schwanzl}), we know that
\[ \ara(I_2(Z))=\ara_h(I_2(Z)) = 2n-1 \]
Moreover, it is known a set of homogeneous generators of $I_2(Z)$ up to radical: using the notation of \cite{bruns-vetter}, set $[i,j]=z_{i0}z_{j1}-z_{j0}z_{i1}$ for $0 \leq i<j \leq n$. Then
\[ I_2(Z) = \sqrt{(\sum_{i+j=k} [i,j] : k=1, \ldots, 2n-1)} \]
(see \cite[Lemma 5.9]{bruns-vetter}).\\
By Remark \ref{8} we have only to add to these generators $F_0=(f_1)_0$ and $F_1=(f_1)_1$ (with the notation of Remark \ref{8}), and so we find $2n+1$ homogeneous polynomials which generate $I$ up to radical.
\end{os}

\begin{thm}\label{10}
Let $X=\mathcal{V}_+(F) \subseteq \mathbb{P}^n$ be a hypersurface such that there exists a line $L \subseteq \mathbb{P}^n$ that meets $X$ only at a point $P$, and let $I$ be the ideal defining the Segre product $X \times \mathbb{P}^1 \subseteq \mathbb{P}^{2n+1}$.
Then
\[ \ara_h(I) \leq 2n \]
\end{thm}

\begin{proof}
By a change of coordinates we can assume that $L=\mathcal{V}_+((x_0, \ldots, x_{n-2}))$. The set $\Omega=\{ [i,j] : i<j, i+j \leq 2n-2 \}$ is an ideal of the poset of the minors of the matrix $Z=(z_{ij})$ (i.e. if $[i,j]\in \Omega$, $h \leq i$ and $h <k\leq j$ then $[h,k]\in \Omega$), so by \cite[Lemma 5.9]{bruns-vetter} 
\[ \ara(\Omega R) \leq \mbox{rank}(\Omega)=2n-2 .\]
We want to prove that $I=\sqrt{J}$ where $J=\Omega R + (F_0,F_1)$ (with the notation of remarks \ref{8}, \ref{9}). To this purpose we may assume that $k$ is algebraically closed by Remark \ref{7}, and we will prove the equivalent condition, by Nullstellensatz and Remark \ref{8}, that $\mathcal{Z}(I_2(Z)+(F_0,F_1))=\mathcal{Z}(J)$.\\
Let be $Q=[Q_0,Q_1]=[Q_{01}, \ldots, Q_{0n},Q_{10}, \ldots, Q_{1n}] \in \mathcal{Z}(J)$. If $Q_0=0$ or $Q_1=0$ trivially $Q\in \Z(I_2(Z))$, so we assume that $Q_0,Q_1$ are points of $\mathbb{P}^n$. First suppose $Q_{ij}\neq 0$ for some $j \leq n-2$ and $i=0,1$: for any $h\neq k$ different from $j$, $[h,j]$ (or $[j,h]$) and $[k,j]$ (or $[j,k]$) are elements of $\Omega$, so since $Q\in \Z(J)$ we easily obtain the relations $Q_{0h}Q_{1k}=Q_{1h}Q_{0k}$, from which $Q \in \mathcal{Z}(I_2(Z))$. We can therefore assume that $Q_{ij}=0$ for all $j<n-1, i=0,1$. But then $Q_0$ and $Q_1$ belong to $ L \cap X$, so $Q_0=Q_1=P$, so $Q\in \Z(I_2(Z))$.
\end{proof}

\begin{cor}\label{ara1}
Let $X \subseteq \mathbb{P}^2$ be a smooth curve of degree $d\geq 3$ such that there exists a line $L \subseteq \mathbb{P}^2$ that meets $X$ only at a point $P$, and let $I$ be the ideal defining the Segre product $X \times \mathbb{P}^1 \subseteq \mathbb{P}^{5}$.
Then
\[ \ara_h(I) =4. \]
Moreover, if $k$ has characteristic $0$, then $\ara(I)=\ara_h(I)=4$.
\end{cor}
\begin{proof}
Theorem \ref{10} implies that $\ara_h(I)\leq 4$. For the lower bound first assume that $k$ is algebraically closed. Since $X$ has positive genus, Proposition \ref{curve} implies that $\ecd(\PP^5 \setminus (X \times \PP^1)) \geq 8$. Thus equation (\ref{ecd}) of Remark \ref{etale} implies that $\ara_h(I)\geq 4$. If $k$ is not algebraically closed, it is obvious that $\ara_h(I)\geq \ara_h(I(P\otimes_k \bar{k}))$, so we have proved the first statement.

If $\chara(k)=0$ Proposition \ref{15} implies that $\cd(P,I)\geq 4$, so $\ara(I)\geq 4$.  
\end{proof}

\begin{os}\label{hyperflexes}
In light of Theorem \ref{10} and Corollary \ref{ara1}, it is natural to define the following set. For every natural numbers $n,d \geq 1$ we define
\[ \V_d^{n-1}=\{ X \subseteq \PP^{n}: X \mbox{ smooth, }\dim X=n-1,  \ \deg X=d, \ \exists \ P \in X \mbox{ as in \ref{10}}\}/ \operatorname{PGL}_n(k)  \]
Notice that all hypersurfaces in $\V_d^{n-1}$ can be represented, by a change of coordinate, by $\V_+(F)$ with $F=x_{n-1}^d + \sum_{i=0}^{n-2} x_i G_i(x_0, \ldots, x_{n})$, where  the $G_i$'s are homogeneous polynomials of degree $d-1$.\\
We start to analyze the case $n=2$, and for simplicity we will write $\V_d$ instead of $\V_d^1$.  So our question is: \emph{How many smooth projective plane curves of degree $d$ do belong to $\V_d$?}

First we list some plane projective curves belonging to $\V_d$.
\begin{compactenum}
\item Every smooth elliptic curve belongs to $\V_3$: in fact every smooth curve of degree greater than or equal to 3 has an ordinary flex, and an elliptic curve meets a line at  most to 3 points. So we recovered \cite[Theorem 1.1]{siwa} as a consequence of Corollary \ref{ara1} (the generators up to radical are different).
\item Obviously, every smooth conic belongs to $\V_2$ too.
\item Every Fermat's curve of degree $d$, i.e.  a projective curve $C = \V_+(F) \subseteq \PP^2$ where $F=x_o^d+x_1^d+x_2^d$, belongs to $\V_d$: in fact one has only to consider the line $\V_+(x_0+ \alpha x_1)$ where $\alpha^d = -1$ and the point $[\alpha,1,0]\in C$, so we recovered also \cite[Theorem 2.8]{song} (the generators are again different).
\end{compactenum}
In their paper \cite[Theorem A]{casnati-del centina}, Casnati and Del Centina compute the dimension of the loci $\mathcal{V}_{d, \alpha}$, $\alpha=1,2$, of all the smooth plane curves of degree $d$ with exactly $\alpha$ points as in Theorem \ref{10}  (if these points are non singular, as in this case, they are called $d$-flexes), and showed that $\mathcal{V}_{d, \alpha}$ are irreducible rational locally closed subvarieties of the moduli space $\mathcal{M}_g$ of curve of genus $g={d-1 \choose 2}$. The dimension of $\mathcal{V}_{d, \alpha}$ is
\[ \dim (\mathcal{V}_{d, \alpha})={d+2-\alpha \choose 2}-8+3 \alpha. \]
Moreover, it is not difficult to show that $\mathcal{V}_{d, 1}$ is an open Zariski subset of $\mathcal{V}_d$, (see \cite[Lemma 2.1.2]{casnati-del centina}), and so
\[ \dim (\mathcal{V}_d)={d+1 \choose 2}-5.  \]
The locus $\mathcal{H}_d$ of all smooth plane curves of degree $d$ up to isomorphism is an open Zariski subset of $\mathbb{P}_{\mathbb{C}}^{{d+2 \choose 2}}$ modulo the group $\operatorname{PGL}_2(\mathbb{C})$, so its dimension is ${d+2 \choose 2}-9$. Then the codimension of $\mathcal{V}_d$ in $\mathcal{H}_d$, provided $d\geq 3$, is $d-3$.\\
So, for example, if we pick a quartic $C$ in the hypersurface $\V_4$ of $\mathcal{H}_4$, Corolloary \ref{ara1} implies that $C\times \PP^1\subseteq \PP^5$ can be defined by exactly four equations. However it remains an open problem to compute the arithmetical rank of $Y\times \PP^1 \subseteq \PP^5$ for any quartic $Y\subseteq \PP^2$. 
\end{os}

In the general case ($n\geq 2$ arbitrary) we can state the following lemma.

\begin{lem}\label{cil}
Let $X\subseteq \PP^n$ be a smooth hypersurface of degree $d$. If $d\leq 2n-3$, or if $d\leq 2n-1$ and $X$ is generic, then $X\in \V^{n-1}_d$.
\end{lem}
\begin{proof}
First we prove the following claim:

a). If $X\subseteq \PP^n$ is a smooth hypersurface of degree $d\leq 2n-1$ not containing lines, then $X\in \V^{n-1}_d$.

We denote by $\Grass(1,n)$ the Grassmannian of lines of $\PP^n$. Consider the projective variety $W_n=\{(P,L)\in \PP^n \times  \Grass(1,n): P\in L\}$. It turns out that this is an irreducible variety of dimension $2n-1$. Now set 
\[T_{n,d}=\{((P,L),F)\in W_n \times L_{n,d}: \ i(L,\V_+(F),P)\geq d\},\]
where by $L_{n,d}$ we denote the projective space of all the homogeneous polynomials of degree $d$ of $K[x_0,\ldots ,x_n]$, and by $i(L,\V_+(F);P)$ the multiplicity intersection of $L$ and $\V_+(F)$ at $P$ (if $L\subseteq \V_+(F)$ then $i(L,\V_+(F);P)=+\infty$). 

Assume that $P=[1,0,\ldots ,0]$ and that $L$ is given by the equation $x_1=x_2=\ldots =x_n=0$. Then it is easy to see that for a polynomial $F\in L_{n,d}$ the condition $(P,L,F)\in T_{n,d}$ is equivalent to the fact that the coefficients of $x_0^d, \  x_0^{d-1}x_1, \ \ldots , \ x_0x_1^{d-1}$ in $F$ are $0$. This implies that $T_{n,d}$ is a closed subset of $\PP^n\times \Grass(1,n)\times L_{n,d}$: thus $T_{n,d}$ is a projective scheme over $k$.

Consider the restriction of the first projection $\pi_1:T_{n,d} \longrightarrow W_n$. Clearly $\pi_1$ is surjective; moreover it follows by the above discussion that all the fibers of $\pi_1$ are projective subspaces of $L_{n,d}$ of dimension $\dim(L_{n,d})-d$. Therefore $T_{n,d}$ is an irreducible projective variety of dimension $2n-1 +\dim(L_{n,d})-d$. 

Now consider the restriction of the second projection $\pi_2: T_{n,d} \longrightarrow L_{n,d}$. Clearly $\V_+(F)\in \V^{n-1}_d$ provided it is smooth, it does not contain any line and it belongs to $\pi_2(T_{n,d})$. So to conclude it is enough to check the surjectivity of $\pi_2$ whenever $d\leq 2n-1$. To this aim, since both $T_{n,d}$ and $L_{n,d}$ are projective, it is enough to show that for a general $F\in \pi_2(T_{n,d})$, the dimension of the fiber $\pi_2^{-1}(F)$ is exactly$2n-1-d$. On the other hand it is clear that the codimension of $\pi_2(T_{n,d})$ in $T_{n,d}$ is at least $d-2n+1$ when $d\geq 2n$. We proceed by induction on $n$ (for $n=2$ we already know this).

First consider the case in which $d\leq 2n-3$. Let $F$ be a general form of $\pi_2(T_{n,d})$, and set $r=\dim(\pi_2^{-1}(F))$. By contradiction assume that $r>2n-1-d$. Consider a general hyperplane section of $\V_+(F)$, and set $F'$ the polynomial defining it. Obviously any element of $\pi_2(T_{n-1,d})$ comes from $\pi_2(T_{n,d})$ in this way, so $F'$ is a generic form of $\pi_2(T_{n,d})$. The condition for a line to belong to a hyperplane is of codimension $2$, so the dimension of the fiber of $F'$ is at least $r-2$. Since $F'$ is a polynomial of $K[x_0,\ldots ,x_{n-1}]$ of degree $d\leq 2(n-1)-1$, we can apply an induction getting $r-2\leq2n-3-d$, so that $r\leq 2n-1-d$, which is a contradiction. 

We end with the case in which $d=2n-1$ (the case $d=2n-2$ is easier). Let $F$ and $r$ be as before, and suppose by contradiction that $r\geq 1$. This implies that there exists a hypersurface $\H\subseteq \Grass(n-1,n)$ such that for any general $H\in \H$ the polynomial defining $\V_+(F)\cap H$ belongs to $\pi_2(T_{n-1,d})$. This implies that the codimension of $\pi_2(T_{n-1,d})$ in $T_{n-1,d}$ is less than or equal to $1$, whereas we know that this is at least $2$.  

So a) holds true. Now we prove the lemma by induction on $n$ (if $n=2$ it is true).

If $d\leq 2n-3$, then we cut $X$ by a generic hyperplane $H$. It turns out (using Bertini's theorem) that $X\cap H \subseteq \PP^{n-1}$ is the generic smooth hypersurface of degree $d\leq 2(n-1)-1$, so by induction there exist a line $L\subseteq H$ and a point $P\in \PP^n$ such that $(X\cap H)\cap L=\{P\}$. So we conclude that $X\in \V_d^{n-1}$.

It is known that the generic hypersurface of degree $d\geq 2n-2$ does not contain lines. So if $d=2n-2$ or $d=2n-1$ the statement follows by a).
\end{proof}

\begin{cor}\label{bello}
Let $X\subseteq \PP^n$ be a smooth hypersurface of degree $d$, and let $I$ be the ideal defining the Segre product $X \times \mathbb{P}^1 \subseteq \mathbb{P}^{2n+1}$. If $d\leq 2n-3$, or if $d\leq 2n-1$ and $X$ is generic, then 
\[ \ara_h(I) \leq 2n \]
\end{cor}
\begin{proof}
Just combine the above lemma with Theorem \ref{10}.
\end{proof}

Putting some stronger assumptions on the hypersurfaces we can even compute the arithmetical rank of the ideal defining their Segre product with $\PP^1$ (and not just give an upper bound as in Theorem \ref{10}). 

\begin{thm}\label{45}
Let $X=\V_+(F) \subseteq \PP^n$ be such that, $F=x_n^d+\sum_{i=0}^{n-3}x_iG_i(x_0,\ldots ,x_n)$ ($G_i$ homogeneous polynomials of degree $d-1$), and let $I$ be the ideal defining the Segre product $X \times \mathbb{P}^1 \subseteq \mathbb{P}^{2n+1}$. Then
\[ \ara_h(I) \leq 2n-1. \]
Moreover, if $X$ is smooth, then
\[ \ara_h(I) = 2n-1. \]
Furthermore, if $k$ has characteristic 0, then $\ara(I)=\ara_h(I)=2n-1$.
\end{thm}
\begin{proof}
We can assume that $k$ is algebraically closed. If $X$ is smooth, by Proposition \ref{14} $\ecd(\PP^{2n+1}\setminus X)\geq 4n-1$, and equation \ref{ecd} yields $\ara_h(I)\geq 2n-1$. If $\chara(k)=0$ Proposition \ref{mah} implies that $\ara(I)\geq 2n-1$.

Now we prove that the upper bound holds.
Consider the set $\Omega=\{ [i,j] : i<j, i+j \leq 2n-3 \}$. As in the proof of Theorem \ref{10}, we have
\[ \ara(\Omega R) \leq \mbox{rank}(\Omega)=2n-3 .\]
Now the proof is completely analogue to that of Theorem \ref{10}.
\end{proof}

\begin{os}
Notice that, if $n\geq 4$, the generic hypersurface of $\PP^n$ defined by the form $F=x_n^d+\sum_{i=0}^{n-3}x_iG_i(x_0,\ldots ,x_n)$ is smooth (whereas if $n\leq 3$ and $d\geq 2$ such a hypersurface is always singular).
\end{os}


The below argument uses ideas from \cite{siwa}. Unfortunately to use these kinds of tools we have to make some restrictions to $\chara(k)$.
\begin{thm}\label{conic}
Assume $\chara(k)\neq 2$. Let $C=\V_+(F)$ be a conic of $\PP^2$, and let $I$ be the ideal defining the Segre product $X=C\times \PP^m \subseteq \PP^{3m+2}$. Then
\[\ara_h(I)=3m.\]
In particular $X$ is a set-theoretic complete intersection if and only if $m=1$. Moreover, if $\chara(k)=0$, then $\ara(I)=\ara_h(I)=3m$. 
\end{thm}
\begin{proof}
First we want to give $3m$ homogeneous polynomials of $S=k[z_{ij}: i=0,1,2, \ \ j=0, \ldots ,m]$ which define $I$ up to radical.\\
For $i=0, \ldots ,m$ choose $F_i$ as in Remark \ref{9}. Then, for all $0\leq j<i \leq m$, set
\[F_{ij}=\sum_{k=0}^2 \displaystyle \frac{\de F}{\de x_k}(z_{0i}, z_{1i} , z_{2i})z_{jk}. \]
Finally we set $G_h=\sum_{i+j=h}F_{ij}$ for all $h=1, \ldots ,2m-1$. We claim that
\[ I=\sqrt{J}, \mbox{ where } J=(F_i,G_j:i=0, \ldots ,m, \ \ j=1,  \ldots ,2m-1). \]
The inclusion $J \subseteq I$ follows from the Euler's formula, since $\chara(k)\neq 2$.\\
We can assume $k$ algebraically closed by Remark \ref{7}, so we have to prove that $I\subseteq \sqrt{J}$, i. e., by the Nullstellensatz, that $\Z(J) \subseteq \Z(I)$. Pick $P\in \Z(J)$, and write $P$ as $P=[P_0,P_1,\ldots ,P_m]$ where $P_j=[P_{0j},P_{1j},P_{2j}]$.
Since  $F_i(P)=0$, for every $i=0, \ldots ,m$ \ $P_i=0$ or $P_i \in C$. So we have to prove that the $P_i$'s that are non zero are equal as points of $\PP^2$.\\ 
By contradiction, let $i$ be the minimum integer such that $P_i \neq 0$ and there exists $k$ such that $P_k \neq 0$ and $P_i \neq P_k$ as points of $\PP^2$, and set $j$ the least among these $k$ (so $i<j$). Set $h=i+j$. We claim that $P_k=P_l$ provided that $k+l=h$, $k<l$, $k\neq i$, $P_k \neq 0$ and $P_l\neq 0$.\\ 
In  fact, if $l<j$, then $P_i=P_l$ by the choice of $j$. But for the same reason also $P_k=P_i$, so $P_k=P_l$. If $l>j$, then $k<i$, so it follows that $P_k=P_l$ by the choice of $i$. So $F_{lk}(P)=0$, because $P_k$ belongs to the tangent of $C$ in $P_l$ (being $P_l=P_k$). Then $G_h(P)=F_{ji}(P)$, and so, since $P \in \Z(J)$, $F_{ji}(P)=0$: this means that $P_i$ belongs to the tangent line of $C$ in $P_j$, which is possible, being $C$ a conic, only if $P_i=P_j$, a contradiction.

For the lower bound, we can assume that $k$ is algebraically closed as in the proof of Corollary \ref{ara1}. By Proposition \ref{14} $\ecd(\PP^{2n+1}\setminus X)\geq 4n-1$, and equation (\ref{ecd}) yields $\ara_h(I)\geq 2n-1$. If $\chara(k)=0$ Proposition \ref{mah} implies that $\ara(I)\geq 2n-1$.
\end{proof}

\begin{os}
B\u adescu and Valla computed recently in \cite{bava}, independently from this work, the arithmetical rank of the ideal defining any rational normal scroll. Since the Segre product of a conic with $\PP^m$ is a rational normal scroll, Theorem \ref{conic} is a particular case of their result.
\end{os}

We end the paper with a proposition that yields a natural question.

\begin{prop}\label{1000}
Let $X=\V_+(F) \subseteq \PP^n$ be a  hypersurface smooth over a field of characteristic $0$ and let $I \subseteq P=k[z_0, \ldots ,z_N]$ be the ideal defining $Z=X \times \mathbb{P}^m \subseteq \mathbb{P}^N$ (any embedding), with $m\geq 1$. Then
\begin{displaymath}
\cd(P,I) = \left\{ \begin{array}{cc} N-1 & \mbox{if \ } n=2 \mbox{ \ and \ }\deg(F)\geq 3\\
N-2 &  \mbox{otherwise} \end{array}  \right.
\end{displaymath}
\end{prop}
\begin{proof}
By Remark \ref{0} we can assume $k=\CC$. Using equation (\ref{betti1}) we have 
\[ \beta_0(Z)=1, \mbox{ \ \ } \beta_1(Z)=\beta_1(X) \mbox{ \ \ and \ \ } \beta_2(Z)=\beta_2(X)+1 \geq 2.\]
If $n=2$, notice that $\beta_1(X)\neq 0$ if and only if $\deg(F)\geq 3$. In fact, by equation (\ref{betti2}), 
\[\beta_1(X)=h^{01}(X)+h^{10}(X)=2h^{01}(X)\]
(the last equality comes from Serre's duality). But $h^{01}(X)$ is the geometric genus of $X$, therefore it is different from $0$ if and only if $\deg(F)\geq 3$. So if $n=2$ we conclude by Theorem \ref{1}.\\
If $n>2$ we have $\beta_1(X)=0$ by the Lefschetz hyperplane theorem \cite[Theorem 3.1.17]{la}, therefore we conclude by Theorem \ref{1}.
\end{proof}

In light of the above proposition, it is natural the following question.

\begin{qs}
With the notation of Proposition \ref{1000}, if we consider the Segre embedding of $Z$ (and so $N=nm+n+m$), do the integers $\ara(I)$ and $\ara_h(I)$ depend only on $n$ ,$m$ and $\deg(F)$?
\end{qs}


\begin{thebibliography}{99}
\addcontentsline{toc}{chapter}{Bibliografia}
\bibitem[Ar]{arap}D. Arapura, \textsl{Algebraic Geometry over the Complex Numbers}, notes on line, 2008.
\bibitem[BaVa]{bava}L. B\u adescu, G. Valla, \textsl{Grothendieck-Lefschetz Theory, Set-Theoretic Complete Intersections and Rational Normal Scrolls}, preprint, 2009.
\bibitem[Bar]{barile}M. Barile, \textsl{Arithmetical Ranks of Ideals Associated to Symmetric and Alternating Matrices}, Journal of Algebra 176, v. 1, 1995.
\bibitem[BS]{bruns-schwanzl}W. Bruns, R. Schw\"anzl, \textsl{The number of equations defining a determinantal variety}, Bull. London Math. Soc. 22, pp. 439-445, 1990.
\bibitem[BrVe]{bruns-vetter}W. Bruns, U. Vetter, \textsl{Determinantal Rings}, Lecture Notes in Mathematics 1327, Springer-Verlag, 1988.
\bibitem[CD]{casnati-del centina}G. Casnati, A. Del Centina, \textsl{On Certain Loci of Smooth Degree $d\geq 4$ Plane Curves with $d$-Flexes}, Michigan Math. J. 50, 2002.
\bibitem[EE]{eisenbudevans}D. Eisenbud, E.G. Evans, \textsl{Every Algebraic Set in n-Space is the Intersection of n Hypersurface}, Inventiones math. 19, 107-112, Springer-Verlag, 1972.
\bibitem[GW]{GW}S. Goto, K. Watanabe, \textsl{On graded rings, I}, J. Math. Soc. Japan 30, n. 2, pp. 179-213, 1978.
\bibitem[Gr1]{gro}A. Grothendieck, \textsl{On the Derham cohomology of algebraic varieties}, Publications math\'ematiques de l' I.H.\'E.S. 29, pp. 95-103, 1966.
\bibitem[Gr2]{gro1}A. Grothendieck, \textsl{Local Cohomology}, Lecture Notes in Mathematics 41 (written by R. Hartshorne), 1967.
\bibitem[Gr3]{sga4}A. Grothendieck, \textsl{Th\'eories des Topos et Cohomologie \'Etale des Sch\'emas}, SGA 4, Lecture Notes in Mathematics 270, Springer-Verlag, 1973.
\bibitem[Har1]{hart2}R. Hartshorne, \textsl{Cohomological Dimension of Algebraic Varieties}, Ann. of Math. 88, pp. 403-450, 1968.
\bibitem[Har2]{hartshorne4}R. Hartshorne, \textsl{Ample Subvarieties of Algebraic Varieties}, Lecture Notes in Mathematics 156, Springer-Verlag, 1970.
\bibitem[Har3]{hart1}R. Hartshorne, \textsl{Varieties of small codimension in projective space}, Bull. of the Amer. Math. Soc. 80, n. 6, 1974.
\bibitem[HS]{ha-sp}R. Hartshorne, R. Speiser, \textsl{Local cohomological dimension in characteristic $p$}, Ann. of Math. 105, n.1, pp. 45-79, 1977.
\bibitem[Hat]{hatcher}A. Hatcher, \textsl{Algebraic Topology}, Cambridge University Press, 2002.
\bibitem[HL]{hu-ly}C. Huneke, G. Lyubeznik, \textsl{On the Vanishing of Local Cohomology Modules}, Invent. Math. 102, pp. 73-93, 1990.
\bibitem[La]{la}R. Lazarsfeld, \textsl{Positivity in Algebraic Geometry I}, A Series of Modern Surveys in Mathematics, Springer, 2004.
\bibitem[Ly1]{ly4}G. Lyubeznik, \textsl{A survey on problems and results on the number of defining equations}, Commutative algebra, Berkley, pp. 375-390, 1989.
\bibitem[Ly2]{ly}G. Lyubeznik, \textsl{The number of defining equations of affine sets}, Amer. J. Math. 114, n. 2, pp. 413-463, 1992.
\bibitem[Ly3]{ly3}G. Lyubeznik, \textsl{\'Etale cohomological dimension and the topology of algebraic varieties}, Ann. of Math. 137, n.1,, pp. 71-128, 1993.
\bibitem[Ly4]{ly2}G. Lyubeznik, \textsl{A Partial Survey of Local Cohomology}, Lecture notes in pure and applied math. 226, pp. 121-154, 2002.
\bibitem[Ly5]{ly1}G. Lyubeznik, \textsl{On the vanishing of local cohomology in characteristic $p>0$}, Compositio Math. 142, pp. 207-221, 2006.
\bibitem[Li]{liu}Q. Liu, \textsl{Algebraic Geometry and Arithmetic Curves}, Oxford Graduate Texts in Mathematics 6, Oxford University press, 2006.
\bibitem[Mi1]{milne}J.S. Milne, \textsl{\'Etale Cohomology}, Princeton University Press, 1980.
\bibitem[Mi2]{milnel}J.S. Milne, \textsl{Lectures on \'etale cohomology}, available on the web, 1998.
\bibitem[Ne]{new}P.E. Newstead, \textsl{Some subvarieties of Grassmannian of codimension 3}, Bull. London Math. Soc. 12, pp. 176-182, 1980.
\bibitem[Og]{ogus}A. Ogus, \textsl{Local cohomological dimension}, Annals of Mathematics, 98, n.2, pp. 327-365, 1973.
\bibitem[PS]{PS}C. Peskine and L. Szpiro, \textsl{Dimension projective finie et cohomologie locale}, Publ. Math. Inst. Hautes \'Etudes Sci. 42, pp. 47-119, 1973.
\bibitem[Sh]{sha}I.R. Shafarevich, \textsl{Basic algebraic geometry 2}, second edition, Springer-Verlag, 1997.
\bibitem[Se]{GAGA}J.P. Serre, \textsl{Geometrie algebrique et geometrie analytique}, Ann. Inst. Fourier, 6, pp.1-42, 1956.
\bibitem[SW]{siwa}A.K. Singh, U. Walther, \textsl{On the arithmetic rank of certain Segre products}, Contemp. Math., 390, pp.147-155, 2005.
\bibitem[So]{song}Q. Song, \textsl{Questions in local cohomology and tight closure}, Unpublished, 2007.
\bibitem[Sp]{spe}R. Speiser, \textsl{Varieties of Low Codimension in Characteristic $p>0$}, Trans. of the Amer. Math. Soc. 240, pp. 329-343, 1978.


\end{thebibliography}
\end{document}